\documentclass[12pt,thmsa, reqno]{amsart}

\let\oldtocsection=\tocsection

\let\oldtocsubsection=\tocsubsection

\renewcommand{\tocsection}[2]{\hspace{0em}\oldtocsection{#1}{#2}}
\renewcommand{\tocsubsection}[2]{\hspace{2em}\oldtocsubsection{#1}{#2}}

\usepackage{setspace}

\usepackage{amssymb, euscript,  mathrsfs}
\usepackage[dvips]{graphics}
\usepackage[title]{appendix}

\input{amssym.def}

\input{amssym.tex}

\usepackage{lineno}

\usepackage{tikz}
\usepackage[utf8]{inputenc}

\usetikzlibrary{matrix,arrows,decorations.pathmorphing}

\usepackage[all,cmtip]{xy}

\usepackage{amsmath}
\let\oldAA\AA
\renewcommand{\AA}{\text{\normalfont\oldAA}}

\def\Cal{\mathcal}

\def\A{{\Cal A}}

\def\C{{\Cal C}}

\def\P{{\Cal P}}

\def\S{{\Cal S}}

\def\bbr{{\Bbb R}}

\def\bbn{{\Bbb N}}

\def\bbc{{\Bbb C}}

\def\bbs{{\Bbb S}}

\def\sgn{{\hbox{\rm sgn}}}

\def\const{{\hbox{\rm const}}}

\def\rn{\bbr^n}

\def\part{\partial}
\def\intl{\int\limits}
\def\b{\beta}

\def\lang{\langle}
\def\rang{\rangle}
\def\Gam{\Gamma}

\def\a{\alpha}
\def\om{\omega}

\def\del{\delta}
\def\vp{\varphi}

\def\gam{\gamma}

\def\Gam{\Gamma}

\def\sig{\sigma}
\def\lam{\lambda}
\def\z{\zeta}

\def\e{\varepsilon}
\def\t{\tau}

\def\chi{{\bf 1}}

\def\snm1{\bbs^{n-1}}

\def\intl{\int\limits}

\def\Cs{\mathscr{C}}

\def\Cs{\mathscr{C c 1234}}

\def\cd{\stackrel{*}{\C}\!{}_{m, k}^\lam}
\def\sd{\stackrel{*}{\S}\!{}_{m, k}^\lam}
\def\cd0{\stackrel{*}{\C}\!{}_{m, k}^\lam}
\def\sd0{\stackrel{*}{\S}\!{}_{m, k}^\lam}

\def\ncd0{\stackrel{*}{\Cs}\!{}_{m, k}^\lam}

\newtheorem{theorem}{Theorem}[section]
\newtheorem{lemma}[theorem]{Lemma}

\theoremstyle{definition}
\newtheorem{definition}[theorem]{Definition}

\theoremstyle{remark}
\newtheorem{remark}[theorem]{Remark}

\numberwithin{equation}{section}

\theoremstyle{corollary}
\newtheorem{corollary}[theorem]{Corollary}
\newtheorem{proposition}[theorem]{Proposition}

\numberwithin{equation}{section}

\newcommand{\be}{\begin{equation}}
\newcommand{\ee}{\end{equation}}

\newcommand{\bea}{\begin{eqnarray}}
\newcommand{\eea}{\end{eqnarray}}
\newcommand{\Bea}{\begin{eqnarray*}}
\newcommand{\Eea}{\end{eqnarray*}}

\def\sideremark#1{\ifvmode\leavevmode\fi\vadjust{\vbox to0pt{\vss
 \hbox to 0pt{\hskip\hsize\hskip1em
\vbox{\hsize2cm\tiny\raggedright\pretolerance10000
 \noindent #1\hfill}\hss}\vbox to8pt{\vfil}\vss}}}%

\begin{document}


\title[On Fractional Integrals]
{On Fractional Integrals Generated by Radon Transforms over Paraboloids}


\author{ B. Rubin}

\address{Department of Mathematics, Louisiana State University, Baton Rouge,
Louisiana 70803, USA}
\email{borisr@lsu.edu}





\begin{abstract}

We apply the Fourier transform technique and a modified version of E. Stein's interpolation theorem  communicated by  L. Grafakos, to obtain sharp $L^p$-$L^q$ estimates for the   Radon transform and more general convolution-type fractional integrals    with the kernels having singularity on the paraboloids.

 \end{abstract}

\maketitle



\section{Introduction}

In the present paper we develop the Fourier transform approach to the   Radon-type transform
\be \label {ppar}
(P f)(x) =\intl_{\bbr^{n-1}} f(x'-y', x_n  -|y'|^2)\, dy' \ee
 and more general fractional integrals
   \be\label{poaxzc}
(P_{\pm}^{\,\a} f)(x) =\frac{1}{\Gam (\a)} \intl_{\bbr^n}(y_n -|y'|^2)_{\pm} ^{\a -1} \,f(x-y)\, dy,\quad Re \, \a > 0,\ee
 which extend  analytically  to $Re \, \a \le  0$ and have a singularity on the paraboloid $y_n=|y'|^2$.
  Here $x=(x', x_n) \in \rn$, $x' \in \bbr^{n-1}$ (similarly for $y$), and  $f$ is a sufficiently good function on $\rn$; see  Section 2 for notation.
   We call (\ref{ppar}) {\it the parabolic Radon transform}. This operator   was  also considered in \cite{Chr, Ru22, Ru22a}. The limiting case $\a=0$ in (\ref{poaxzc}) yields (\ref{ppar}).

 Sharp $L^p$-$L^q$ estimates for $P f$ can be obtained  by making use  of the  Oberlin-Stein theorem  for the usual Radon transform over affine hyperplanes in $\rn$ \cite {OS}.  They can also be derived from the similar estimates for  the  transversal Radon transform
 \be \label {bart}
(Tf)(x)=\intl_{\bbr^{n-1}} f(y', x_n + x'\cdot y')\, dy',\ee
 which was introduced by Strichartz \cite {Str} in his study of Radon transforms on the Heisenberg group; see   \cite{Ru12}, \cite [Section 4.13] {Ru15}, \cite{Ru22} for details.

  The purpose of the present  paper  is to present a straightforward Fourier transform approach to the $L^p$-$L^q$ estimates for the entire analytic family $\{ P_{\pm}^{\,\a}\}_{\a \in \bbc}$, taking into account that these operators have convolution structure. 
  
 This approach becomes possible thanks to a  version
of Stein’s interpolation theorem \cite{Ste56}   communicated by L. Grafakos \cite {Graf}. An advantage of this version is that the hypotheses of Stein's theorem, 
which make this theorem  applicable, are presented in \cite {Graf} in the more convenient terms of smooth compactly supported functions,  rather than in terms of   simple functions in \cite {Ste56}. We recall that simple functions are finite  linear combinations of the characteristic functions of disjoint compact sets;
see Section \ref{Inter} for details.  A close interpolation theorem in the multi-linear setting  was recently proved  by  Grafakos and Ouhabaz \cite{GO}.

 Some $L^p$-$L^q$ estimates for  localized modifications of (\ref{poaxzc}) with a
 smooth  cut-off function under the sign of integration were announced  by
 Littman \cite {Litt} and Tao \cite {Tao}, who referred to Stein's interpolation theorem  \cite {Ste56}.  A  distinctive feature of  $P_{\pm}^{\,\a} f$ in comparison with  \cite {Litt, Tao} is that our operators are not localized and their Fourier transforms can be explicitly computed. A roundabout approach to the study of operators (\ref{poaxzc}) via non-convolution-type fractional integrals generated by the operator $T$ was developed in \cite{Ru22a}.
 
It is interesting to note that implementation of the Fourier transform  shows that the analytic continuations of $ P_{\pm}^{\,\a}$ at $\a =(1-n)/2$ extend as unitary operators in $L^2 (\rn)$. This fact plays an important role  in the  inverse problems for elastic wave equations \cite{BK, NR}.

 To formulate our  main result, let  $S(\bbr^n)$ be the Schwartz space of infinitely differentiable functions on $\rn$, which are rapidly
decreasing together with their derivatives of all orders.

\begin {theorem} \label {lanseT} Let $1\le p,q \le \infty$; $\a_0=Re\, \a$. The operators $P_{\pm}^{\,\a}$,  initially defined on functions $\vp \in S (\rn)$ by analytic continuation, extend as linear bounded operators  from $L^p (\rn)\!$ to $\!L^q (\rn)$ if and only if
\be\label {wz2as}
\frac{1-n}{2} \le \a_0 \le 1, \qquad p=\frac{n+1}{n + \a_0}, \qquad q=\frac{n+1}{1-\a_0}.\ee
In particular, the Radon transform $P$ is bounded from $L^p (\rn)\!$ to $\!L^q (\rn)$ if and only if $p=(n+1)/n$ and $q=n+1$.
\end{theorem}

Although this statement is not new (cf. \cite[Theorem 1.1]{Ru22a}), our method of the proof, which relies on explicit computation of the Fourier
transform of $P_{\pm}^{\,\a} f$ combined with Grafakos' version of the interpolation theorem, might be of interest and instructive.

\noindent{\bf Plan of the Paper.}  Sections 2 and 3 contain notation and elementary properties of the convolution operators $P_{\pm}^{\,\a}$.
In Section 4 we  compute their Fourier multipliers in the framework of the corresponding distribution theory.
Section 5  contains auxiliary material  and detailed proof of Theorem \ref{lanseT}. Section 6  provides some comments and indicates possible  generalizations.  Some technical calculations are moved to Appendix.

\section{Notation}

 In the following, $x\!=\!(x_1, \ldots, x_{n-1}, x_n)\!=\!(x', x_n) \!\in \!\rn$.
The notation $C(\bbr^n)$,  $C^\infty (\bbr^n)$,
and $L^p (\bbr^n)$ for function spaces is standard;   $||\cdot ||_p =||\cdot ||_{L^p (\bbr^n)}$;
 $C_c^\infty (\bbr^n)$ is the space of compactly supported infinitely differentiable functions on $\rn$. The notation $\langle f,g\rangle$ for  functions $f$ and $g$ is used for the integral of the product of these functions.
 We keep the same notation  when $f$ is a distribution and $g$ is a test function.

 If $m=(m_1, \ldots,  m_n)$ is a multi-index, then $\partial^m = \partial_1^{m_1} \ldots     \partial_n^{m_n}$,  where  $\partial_i  =\partial/ \partial x_i$.
 The Fourier transform of a function
$f \in  L^1 (\bbr^n)$ is defined by
\be \label{ft} \hat f (\xi) = \intl_{\bbr^{ n}} f(x) \,e^{ i x \cdot \xi} \,dx, \qquad \xi\in \rn,
\ee where $x \cdot \xi = x_1 \xi_1 + \ldots + x_n\xi_n$.
 We denote by $S(\bbr^n)$
the Schwartz space of $C^\infty$-functions
which are rapidly
decreasing together with their derivatives of all orders. The space $S(\bbr^n) $ is equipped  with the topology generated by the sequence
of norms
\be\label{setop} ||\vp||_k=\max\limits_x(1+|x|)^k
\sum_{|j|\le k} |(\part^j\vp)(x)|, \quad k=0,1,2, \ldots .\ee

The Fourier transform is an automorphism of $S(\bbr^n)$.
The space of tempered distributions, which is dual to $S(\bbr^n)$, is denoted by $S'(\bbr^n)$.  The Fourier transform of a distribution $f\in S'(\bbr^n)$ is  a
distribution $\hat f\in S'(\bbr^n)$ defined by
\be\label{ftrd12}
\lang \hat f,\psi \rang =\lang f,\hat \psi \rang,\qquad \psi\in S(\bbr^n). \ee
The equality (\ref{ftrd12}) is equivalent to
\be\label{ftrd12V}
\lang \hat f,\hat\vp \rang =(2\pi)^n \lang f,\vp_1 \rang,\qquad \vp\in S(\bbr^n), \quad \vp_1 (x)=\vp (-x). \ee
The inverse Fourier transform of a function (or distribution) $f$ is denoted by $\check f$.

Given a  real-valued quantity
$X$ and a complex number $\lam$, we set $(X)_\pm^\lam= |X|^\lam$ if $\pm X>0$ and $(X)_{\pm}^\lam=0$, otherwise.
 All integrals are understood in the  Lebesgue sense.  The letter $c$, sometimes with subscripts, stands for a nonessential constant that may be different at each
occurrence.

\section {Elementary Properties of Parabolic Convolutions  $P_{\pm}^{\,\a} f$}

For $y=(y_1, \ldots, y_{n-1}, y_n)=(y', y_n) \in \rn$, we denote
\bea
p_{\a +}(y)&=&(y_n -|y'|^2)_{+} ^{\a -1}=\left \{
\begin{array} {ll}  (y_n -|y'|^2)^{\a -1} &\mbox{ if $y_n >|y'|^2$},\\
 0 &\mbox{ otherwise},\\
 \end{array}
\right. \nonumber\\
p_{\a -}(y)&=&(y_n -|y'|^2)_{-} ^{\a -1}=\left \{
\begin{array} {ll}  (|y'|^2 -y_n )^{\a -1} &\mbox{ if $y_n <|y'|^2$},\\
 0 &\mbox{ otherwise}.\\
 \end{array}
\right.\nonumber\eea
These functions have a singularity on the paraboloid $y_n=|y'|^2$.
The corresponding convolution operators
\be\label{poaxz}
(P_{\pm}^{\,\a} f)(x) = \intl_{\bbr^n} p_{\a \pm}(y)f(x-y)\, dy, \qquad Re \,\a >0,\ee
are generalizations of the   Riemann-Liouville fractional integrals \cite{SKM}
\be\label{poqa}
(I_{\pm}^{\,\a} f)(t)=\frac{1}{\Gam (\a)}\intl_{\bbr}  s_{\pm}^{\a -1} f(t-s) \,ds, \qquad t\in \bbr,\ee
and coincide with them if $n=1$. We call (\ref {poaxz})   {\it the parabolic  fractional integrals}.

\begin{lemma} \label {side} Let  $f \!\in S (\rn)$.  The following statements hold.

\noindent {\rm (i)} For each $x \in \rn$, $(P_{\pm}^{\,\a} f)(x)$
 extend as  entire functions of $\a$. Moreover,
\be\label {Daca} \lim\limits_{\a \to 0} (P_{\pm}^{\,\a} f)(x) =(Pf)(x),\ee
where $(P f)(x)$ is the parabolic Radon transform (\ref{ppar}).

\noindent {\rm (ii)}  If $Re \,\a  > 0$, then  for any multi-index $m$,
\be\label {Dacw} \partial^m P_{\pm}^{\,\a} f =P_{\pm}^{\,\a} \partial^m f.\ee

\noindent {\rm (iii)} For any positive integer $k$,
\be\label {Ded}
 P_{\pm}^{\,\a} f = (\pm 1 )^{k} P_{\pm}^{\,\a +k} \partial_n^k  f= (\pm 1 )^{k} \partial_n^k P_{\pm}^{\,\a +k} f. \ee
\end{lemma}
\begin{proof} To prove {\rm (i)}, we have
 \bea
(P_{\pm}^{\,\a} f)(x)\!\!&=&\!\!\frac{1}{\Gam (\a)}\intl_{\bbr^{n-1}} dy'\intl_{-\infty}^\infty  (y_n -|y'|^2)_{\pm}^{\a -1}  f (x' -y', x_n -y_n)\, dy_n \quad \nonumber\\
\label {008d}&=&\!\!\frac{1}{\Gam (\a)}\intl_0^\infty \!s ^{\a -1} A_{x,\pm} (s)\,ds; \\
A_{x,\pm} (s)\!\!&=&\!\!\intl_{\bbr^{n-1}} \!\!\! f (x' -y', x_n  - |y'|^2 \mp s)\, dy'.\nonumber\eea
Because the functions $A_{x,\pm} (s)$ are smooth, rapidly decreasing, and satisfy $A_{x,\pm} (0)=(Pf)(x)$, the result follows; cf. \cite [Chapter I, Section 3.2]{GS1}, \cite [Section 2.5]{Ru15}.
 In  {\rm (ii)} we simply differentiate under the sign of integration.  The first equality in (\ref {Ded})  can be obtained using integration by parts. The second equality is the result of differentiation:
 \[
 (\pm 1 )^{k} \partial_n^k P_{\pm}^{\,\a +k} f= (\pm 1 )^{k} \partial_n^k I_{\pm}^{k}P_{\pm}^{\,\a} f=P_{\pm}^{\,\a} f.\]
\end{proof}
\begin{remark}\label {Dacqws} The formula (\ref{Ded}) gives an explicit  expression of the   analytic continuation  of $P_{\pm}^{\,\a} f$ from the domain $Re \,\a  > 0$ to  $Re \,\a  > - k$. For example, for any positive integers $\ell$ and $k>\ell$, setting $\a=-\ell$, we have
\be\label {Dacws}
P_{\pm}^{\,-\ell} f = (\pm 1 )^{k} \partial_n^k P_{\pm}^{\,k-\ell} f.\ee
Note also that by Fubini's theorem,
\be\label{poeqad}
 P_{\pm}^{\,\a} I_{\pm}^{\,\b} f =  I_{\pm}^{\,\b} P_{\pm}^{\,\a} f = P_{\pm}^{\,\a +\b}f, \qquad Re \,\a> 0, \quad  Re \,\b > 0,\ee
where $I_{\pm}^{\,\b}$ are applied in the last variable. Similarly,  for $Re \,\a >0$,
\be\label{psoeqa}
 P_{\pm}^{\,\a} f =  I_{\pm}^{\,\a} P f,\ee
 where $P$ is the Radon transform (\ref{ppar}).  Here it is assumed that $f$ is good enough, so that the change of the order of integration is well justified.
\end{remark}

\subsection{The Dual Fractional Integrals}\label {kiaq}

For $y=(y', y_n) \in \rn$ and $Re \, \a >0$,  denote
\be\label{ped1}
p^*_{\a \pm}(y)=\frac{1}{\Gam (\a)} (y_n +|y'|^2)_{\pm}^{\a -1}= p_{\a \mp}(-y).\ee
The set of singularities of this function is the  paraboloid $y_n=-|y'|^2$ in the negative half-space $y_n \le 0$.
The corresponding convolutions
\be\label{poax}
\stackrel{*}{P}{\!}^\a_{\pm} f = p^*_{\a \pm} \ast f\ee
are called  the {\it dual parabolic  fractional integrals}. The name is motivated by the following lemma.
\begin{lemma}\label {imit} For $Re \,\a >0$,
\be\label{krf1} \lang P^{\,\a}_{\pm} f, \vp \rang =\lang f, \stackrel{*}{P}{\!}^\a_{\mp} \vp \rang,\ee
 provided that either side of this equality exists in the Lebesgue sense.
\end{lemma}
\begin{proof}
\bea
&&l.h.s.=\!\intl_{\rn} \!\vp (x) dx\!\intl_{\bbr^{n-1}}\!\! dy'\!\intl_{-\infty}^{\infty}(y_n -|y'|^2)_{\pm}^{\a -1} f(x'-y', x_n -y_n) \,dy_n\nonumber\\
&&=\intl_{\bbr^{n-1}}\!\! dz'\! \intl_{-\infty}^{\infty} \!f(z', z_n) d z_n\!\intl_{\bbr^{n-1}}\!\! dx'\!\intl_{-\infty}^{\infty}\!\! \vp (x', x_n) (x_n \!- \!z_n \!-\!|x'\! -\!z'|^2)_{\pm}^{\a -1} dx_n\qquad \nonumber\\
&&=\intl_{\rn} f(z) dz\!\intl_{\bbr^{n-1}}\!\! dy'\!\intl_{-\infty}^{\infty} \!\vp (z' \!-\!y', z_n\! -\!y_n) (-y_n \!-\!|y'|^2)_{\pm}^{\a -1} dy_n= r.h.s.\nonumber\eea
\end{proof}

Clearly,
\be\label{pzax}
\stackrel{*}{P}{\!}^\a_{\pm}= JP^{\,\a}_{\mp} J, \qquad (Jf)(x)=f(-x).\ee
An analogue of (\ref {Daca}) for $f \in S(\rn)$ is
\be\label {Dacs} \lim\limits_{\a \to 0} (\stackrel{*}{P}{\!}^\a_{\pm} f)(x) = (P^*f)(x),\ee
where $P^* f$ is the {\it dual parabolic Radon transform} defined by
\be\label {Dacq}
(P^*f)(x)=\intl_{\bbr^{n-1}} f(x'-y', x_n +|y'|^2)\, dy';\ee
cf. (\ref{ppar}).
By Fubini's theorem,
\be\label{poeqaq}
\stackrel{*}{P}{\!}^\a_{\pm} I_{\pm}^{\,\b} f =  I_{\pm}^{\,\b} \stackrel{*}{P}{\!}^\a_{\pm} f = \stackrel{*}{P}{\!}^{\a +\b}_{\pm} f, \qquad Re \,\a >0, \quad  Re \,\b >0,\ee
 if $f$ is good enough; cf. (\ref{poeqad}).  Similarly, for $Re \,\a >0$,
\be\label{psoeqa1}
\stackrel{*}{P}{\!}^\a_{\pm} f=I_{\pm}^{\,\a} P^* f.\ee

\section{Parabolic Convolutions and Distributions}


In the following, we invoke  the  theory of distributions and the Fourier transform. Formally,
\be\label{pyy}
(p_{\a \pm} \ast f)^\wedge =\hat p_{\a \pm} \hat f.\ee
Our aim is to give  this equality precise meaning and compute $\hat p_{\a \pm}$.

\subsection{Some Preparations}

\begin{lemma}\label {ngr} If $Re \, \a >0$, then $p_{\a \pm}$ can be viewed as regular   tempered distributions, which extend as entire distribution-valued functions of $\a \in \bbc$. In particular, for any integer $k>0$ and  $\vp \in S(\rn)$,
\be\label {Dacw1}
\lang p_{\,\a \pm}, \vp \rang  = (\mp 1 )^{k} \intl_{\rn} p_{\,(\a +k)\pm} (y) \, (\partial_n^k  \vp)(y)\, dy, \qquad Re \,\a  > - k.\ee
 In the case  $\a =0$ we have
\[
p_{0 +}(y)= p_{0 -}(y)\stackrel {\rm def}{=}\del_P (y),\]
 where
\be\label {Dirac} \lang \del_P, \vp \rang =\intl_{\bbr^{n-1}} \vp (y', |y'|^2)\, dy'.\ee
\end{lemma}
\begin{proof} If $Re \, \a >0$, then $(p_{\a \pm}, \vp)$ are absolutely convergent integrals that can be estimated by  norms (\ref{setop}); see  Appendix.
Furthermore, as in (\ref {008d}),
\[
\lang p_{\a \pm }, \vp \rang=\frac{1}{\Gam (\a)}\intl_0^\infty \!s ^{\a -1} A_{\pm} (s)\,ds, \quad A_{\pm} (s)\!=\!\!\intl_{\bbr^{n-1}} \!\!\! \vp (y', |y'|^2 \pm s)\, dy',\]
and therefore $p_{\a \pm} $ extend analytically  as  $S'$-distributions to all $\a \in \bbc$. The equality (\ref{Dacw1}) is a consequence of the integration by parts. In the case $\a =0$, as in (\ref{Daca}), we have
\[
 \lang p_{0 \pm}, \vp \rang =\lim\limits_{\a \to 0} \lang p_{\a \pm}, \vp \rang = \intl_{\bbr^{n-1}} \vp (y', |y'|^2)\, dy' = \lang\del_P, \vp \rang.\]
\end{proof}

 Note  that  $\del_P (y)$ differs from the delta distribution $\del (y_n -|y'|^2)$, which is  defined by a surface integral
\be\label {Diracs}
\lang \del (y_n -|y'|^2), \vp \rang =\intl_{y_n =|y'|^2} \vp (y) \, d\sig (y)\ee
and expresses through the integral over $\bbr^{n-1}$ with the corresponding Jacobian factor.

\begin{corollary} For any $\a\in \bbc$ and any positive integer $\ell$,
\be\label {Diracs1}
 (\pm 1 )^{\ell}\partial_n^{\ell} p_{\,\a \pm}= p_{\,(\a-\ell) \pm}.\ee
\end{corollary}
\begin{proof} Suppose that $Re \, \a > -k$, $k \in \bbn$, and let $\vp \in S(\rn)$. By (\ref{Dacw1}),
\bea
 \lang (\pm 1 )^{\ell}\partial_n^{\ell} p_{\,\a \pm}, \vp \rang &=&  (\mp 1 )^{\ell} \lang  p_{\,\a \pm}, \partial_n^{\ell}\vp \rang\nonumber\\
&=&  (\mp 1 )^{\ell +k}  \intl_{\rn} p_{\,(\a +k)\pm} (y) \, (\partial_n^{\ell +k}  \vp)(y)\, dy.\nonumber\eea
The same expression can be obtained for $\lang p_{\,(\a-\ell) \pm}, \vp \rang$ if we replace $\a$ by $\a -\ell$ and $k$ by $\ell +k$ in (\ref{Dacw1}).
\end{proof}

\begin{definition} \label {Digf} {\it Following Lemma \ref{ngr}, we define $P_{\pm}^{\,\a} \vp$ for any $\a \in \bbc$   as  convolutions of the $S'$-distributions $p_{\,\a \pm}$ with the test function $\vp \in S(\rn)$ by the formula}
\be\label {w3a1}  (P_{\pm}^{\,\a} \vp)(x) = \lang p_{\,\a \pm} (y), \vp (x - y)\rang .\ee
\end{definition}

\begin{lemma} \label {Digf1} Let $\a \in \bbc$ and  $\vp \in S(\rn)$. The following statements hold.

\noindent {\rm (i)} $(P_{\pm}^{\,\a} \vp)(x)$ are infinitely differentiable tempered functions on $\rn$. In particular, for any multi-index $m$,
\be\label {w3a2}
(\partial^m P_{\pm}^{\,\a} \vp)(x) =  \lang p_{\,\a \pm} (y),  \partial ^m_x \vp (x - y)\rang.\ee

\noindent {\rm (ii)} For any positive integer $\ell$,
 \be\label {w3as1}
 (\pm \partial_n)^{\ell} P_{\pm}^{\,\a} \vp=P_{\pm}^{\,\a-\ell} \vp.\ee
\end {lemma}
 \begin{proof} The statement (i) mimics known facts for $S'$-convolutions;  see, e.g., \cite [p. 26, Lemma 2.5]{Es}, \cite [p. 84]{Vl}.
To prove (ii), by making use of (\ref{w3a1}) and (\ref{Diracs1}),
we have
\bea
 (\pm \partial_n)^{\ell} (P_{\pm}^{\,\a} \vp)(x)&=& (\pm 1 )^{\ell}  \lang p_{\,\a \pm} (y), (\partial/\partial x_n)^\ell \vp (x - y)\rang \nonumber\\
&=& (\mp 1 )^{\ell}  \lang p_{\,\a \pm} (y), (\partial/\partial y_n)^\ell \vp (x - y)\rang\nonumber\\
&=&  (\pm 1 )^{\ell} \lang \partial_n^{\ell} p_{\,\a \pm}, \vp (x - y)\rang= \lang p_{\,(\a-\ell) \pm}, \vp (x - y)\rang\nonumber\\
&=& (P_{\pm}^{\,\a-\ell} \vp)(x). \nonumber\eea
\end {proof}

All results of this subsection can be easily reformulated for the dual fractional integrals introduced in Subsection \ref{kiaq}.  Because the functions $p^*_{\a \pm}(y)$ differ from $p_{\a \pm} (y)$ only by a sign of $y_n$, they can be viewed as $S'$-distributions, which extend analytically to all $\a \in \bbc$. In particular, for $\a =0$,
\be\label{kss1}
p^*_{0 +}(y)= p^*_{0 -}(y)\stackrel {\rm def}{=}\del^*_p (y),\ee
where
\be\label {Diracd} \lang \del^*_p, \vp \rang=\intl_{\bbr^{n-1}} \vp (y', -|y'|^2)\, dy', \qquad \vp \in S (\rn).\ee

\subsection{The Fourier Transforms $\hat p_{\a \pm}$ and Spaces of the Semyanistyi Type}
 Following \cite[pp.  98, 137]{SKM}, we fix the branches of the analytic functions $(\pm ix_n)^{\lam}$, $\lam \in \bbc$,  by setting
\be\label {mza}
(\pm ix_n)^{\lam}=  \exp \left( \lam \log |x_n| \pm \frac{\lam \pi i}{2}\, \sgn \,x_n \right ),\ee
and denote
\[
\om_+ (x)\!=\! \pi^{(n-1)/2}\exp \left(\!-\frac{i |x'|^2}{4x_n}\right ),  \quad \om_- (x)\!=\!\om_+ (x) \exp \left(\!\frac{(n\!-\!1) \pi i}{2}\, \sgn \,x_n\!\right ),\]
\be\label {bdqa}
q_{\a \pm} (x)\!=\! (\mp ix_n)^{-\a-(n-1)/2} \om_{\pm} (x).\ee
Our aim is to show that
\be\label {yqzjr} \hat p_{\a \pm}=q_{\a \pm}\ee
in a suitable sense.

By the definition of the Fourier transform in $S'(\rn)$,  $\lang \hat p_{\a \pm}, \psi \rang= \lang p_{\a \pm}, \hat\psi \rang$, where $\psi \in S(\rn)$ and $\lang p_{\a \pm}, \hat\psi \rang$ is  understood as analytic continuation of the integral
$\int_{\rn}   p_{\a \pm} (y) \hat\psi (y) dy$ from the domain $Re \, \a >0$.  Thus,  to prove (\ref {yqzjr}), one may be tempted to show that
\[ \intl_{\rn}   p_{\a \pm} (y) \hat\psi (y) dy = \intl_{\rn}   q_{\a \pm} (x) \psi (x) dy\]
for some $Re \, \a >0$, for which both integrals are absolutely convergent.  However, these integrals may not exist simultaneously, rather than $\psi (x)$ vanishes at  $x_n = 0$. Note also that multiplication by $q_{\a \pm}$  does not preserve the space $S(\rn)$.

 To circumvent this obstacle, one can reduce the  space of the test functions in a suitable way and expand the corresponding  space of distributions.
An idea of this approach originated in the work of Semyanistyi \cite{Sem} who treated Fourier multipliers having singularity at a single point $x=0$. This idea was independently used by Helgason \cite [p. 162]{H65}. It was
extended by Lizorkin \cite {Liz} and later by Samko \cite{Sam77} to more general sets of singularities; see also \cite{Ru15, Sam,  SKM} and references therein.  In our case, described below,
the set of singularities is the  hyperplane $x_n=0$.

\begin{definition}\label {Semyanistyi}
{\it We denote by $\Psi(\bbr^n)$
   the subspace of functions $\psi \in S(\bbr^n)$
vanishing with all derivatives on the hyperplane $x_n=0$.
 Let also  $\Phi(\bbr^n)$  be the
  Fourier image of $ \Psi (\bbr^n)$.}\footnote{ In the cited literature, exept \cite{H65},  the notation $\Phi, \Psi$ is used in the case of singularity at a single point $x=0$. For the sake of simplicity, we keep the same notation, which does not cause any ambiguity.}
 \end{definition}


\begin{proposition} \label{izsh} {\rm (see, e.g.,  \cite[Theorem 2.33]{Sam})}  The space $\Phi (\bbr^n)$ is dense in $L^p (\rn)$, $1<p<\infty$.
 \end{proposition}

Every function $f\in L^p (\rn)$, $1\le p\le \infty$, can be regarded as a regular  $\Phi'$-distribution.

\begin{proposition} \label{iqah}  {\rm (cf.  \cite[Lemma 3.11]{Ru96}, \cite{Yos})}  If $f\in L^p (\rn)$ and  $g \in L^q (\rn)$ , $(1\le p,q < \infty)$, coincide
as $\Phi'$-distributions, then they coincide almost everywhere on $\rn$.
 \end{proposition}

Recall that  a function $\mu$ is called a multiplier on a
linear topological space $X$ if the map \[ X \ni f \to \mu f \in X\] is
continuous in the topology of $X$.
  One can readily see that multiplication  by $q_{\a \pm}$ is an automorphism of  $ \Psi (\bbr^n)$.   Note also that the space $\Phi (\bbr^n)$ is invariant under translations, that is,
 for every $h \in \bbr^n$, the map \[ \Phi  (\bbr^n) \ni \vp (x)
\to \vp (x-h)\in \Phi (\bbr^n)\] is continuous in the induced topology of $\Phi (\bbr^n)$.
Indeed, since the Fourier transform maps $\Phi (\bbr^n)$ onto $\Psi (\bbr^n)$
isomorphically, this statement follows from the observation that the
function  $\xi \to
e^{ih\cdot \xi}$ is a multiplier on $ \Psi (\bbr^n)$.

By Theorem 1 from \cite[Chapter III, Section 3.7, p. 148] {GS2}, the inverse Fourier transform $\check q_{\a \pm} (x)\in \Phi'  (\bbr^n)$ is a convolutor on $\Phi (\bbr^n)$, that is the map $ \vp \to \check q_{\a \pm} \ast \vp$ is continuous in $\Phi (\bbr^n)$, and
\be\label {mrza}
(\check q_{\a \pm} \ast \vp)^{\wedge} =q_{\a \pm} \hat \vp.\ee

Note also that the  Riemann-Liouville operators   (\ref{poqa}), acting in the last variable and corresponding to the Fourier multipliers $(\mp ix_n)^{-\a}$, are automorphism of  $ \Phi (\bbr^n)$.

  \begin{lemma} \label {omor} Let $\psi\in \Psi (\bbr^n)$, $\a \in \bbc$. Then
 \be\label {zaza} \lang \hat p_{\a \pm}, \psi \rang= \lang q_{\a \pm}, \psi \rang, \ee
 which means that   $\hat p_{\a \pm}=q_{\a \pm}$ in the $\Psi'$-sense. In particular, if $\a=0$, for  the  delta distribution
(\ref{Dirac}) we have
\be\label {zyqz} \lang \hat\del_P, \psi \rang= \lang q_0, \psi \rang,\ee
where
\be\label {zyqz1} q_0 (x)= q_{0 \pm}(x)= (-ix_n)^{-(n-1)/2} \pi^{(n-1)/2} \exp \left(\!-\frac{i |x'|^2}{4x_n}\right ).\ee
\end{lemma}
\begin{proof} Let us prove (\ref{zaza}) for $p_{\a +}$.
Denote
\be\label {laq}
e_{\e} (y)=  e^{-\e y_n}, \qquad \e >0,\ee
 and suppose first that $\a$ is real-valued. Then, for $\a >0$,
 \be\label {hyt}
\lang \hat p_{\a +}, \psi\rang\stackrel{\rm def }{=} \lang p_{\a +}, \hat\psi \rang= \lim\limits_{\e \to 0} \,\lang p_{\a +} e_\e, \hat\psi \rang=  \lim\limits_{\e \to 0} \,\lang (p_{\a +} e_\e)^\wedge, \psi \rang.\ee
Here the integral $\lang p_{\a +}, \hat\psi \rang$ is absolutely convergent (cf. the proof of Lemma \ref{ngr}), so that application of the Lebesgue dominated convergence theorem is well justified. Note also that
 $p_{\a +} e_\e \in L^1 (\rn)$ because
\bea\intl_{\rn} p_{\a +} (y)\, e^{-\e y_n} dy_n&=&\intl_{\bbr^{n-1}} dy' \intl_{|y'|^2}^\infty (y_n -|y'|^2)^{\a -1} e^{-\e y_n} dy_n\nonumber\\
&=& \intl_{\bbr^{n-1}} e^{-\e |y'|^2 } dy'  \intl_0^\infty s^{\a -1} e^{-\e s}  ds <\infty.\nonumber\eea
 Furthermore, as above,
\bea(p_{\a +} e_\e)^\wedge (x)   &=&\frac{1}{\Gam (\a)}\intl_{\bbr^{n-1}} e^{i x'\cdot y'}  dy'\intl_{|y'|^2}^\infty  e^{i x_n  y_n} (y_n -|y'|^2)^{\a -1} e^{-\e y_n} dy_n\nonumber\\
 &=&\frac{1}{\Gam (\a)}\intl_{\bbr^{n-1}} e^{i x'\cdot y'} e^{- |y'|^2 (\e - ix_n)} dy' \intl_0^\infty  s^{\a -1} e^{-s(\e - ix_n)} ds.\nonumber\eea
 Both integrals can be explicitly evaluated, and we get
 \be\label {hytq}
(p_{\a +} e_\e)^\wedge (x)\!=\! \pi^{(n-1)/2} (\e\! -\!ix_n)^{-\a- (n-1)/2}\, \exp\left (\!-\frac{|x'|^2}{4(\e \!-\!ix_n)}\right ).\ee
Here the power function is defined by
\be\label{power}
z^\lam =e^{\lam (\log |z| +i \arg  z)}, \qquad  -\pi/2  <\arg z <\pi/2,\ee
\[z= \e - ix_n, \qquad \lam = -\a- (n-1)/2.\]
To justify the passage to the limit  in the last expression in (\ref{hyt}), we take into account that $\psi (x)$ has a strong decay as $x_n  \to 0$ and
\bea
 |(p_{\a +} e_\e)^\wedge (x)|&=& \frac{\pi^{(n-1)/2}}{(\sqrt{\e^2+ x_n^2})^{\a+(n-1)/2}}\, \exp\left (\!-\frac{\e|x'|^2}{4(\e^2+x_n^2)}\right )\nonumber\\
 &\le& \frac{\pi^{(n-1)/2}}{|x_n|^{\a+(n-1)/2}}.\nonumber\eea
 Thus for $\a >0$,  owing to (\ref {mza}), we obtain
\bea \label {hytq1} &&\lim\limits_{\e \to 0} \,\lang (p_{\a +} e_\e)^\wedge, \psi \rang= \lang \lim\limits_{\e \to 0} \,  (p_{\a +} e_\e)^\wedge, \psi \rang\\
&&=\pi^{(n-1)/2} \intl_{\rn } (-ix_n)^{-\a-(n-1)/2} \exp \left(-\frac{i |x'|^2}{4x_n}\right )  \psi (x)\, dx= \lang q_{\a +}, \psi \rang.\nonumber\eea
 Reverting  to (\ref{hyt}), for $\a>0$ we obtain
$\lang \hat p_{\a +}, \psi \rang= \lang q_{\a +}, \psi \rang$.
Noting that $\lang \hat p_{\a +}, \psi \rang= \lang p_{\a +}, \hat\psi \rang$ is an entire function of $\a$ and the same is true for $\lang q_{\a +}, \psi \rang$, we make use of the analytic continuation and conclude that $\lang \hat p_{\a +}, \psi \rang= \lang q_{\a +}, \psi \rang$  for all $\a \in \bbc$.

 To compute the Fourier transform of $ p_{\a -}$, unlike (\ref{laq}), we set
\[
\eta_{\e} (y)=  e^{\e y_n-2\e|y'|^2}, \qquad \e >0.     \]
Note that $p_{\a -} \eta_\e \in L^1 (\rn)$. Indeed,
\bea\intl_{\rn} p_{\a -} (y)\, \eta_{\e} (y) dy_n&=&\intl_{\bbr^{n-1}}  e^{-2\e|y'|^2} dy' \intl^{|y'|^2}_{-\infty} (|y'|^2 -y_n)^{\a -1} e^{\e y_n} dy_n\nonumber\\
&=& \intl_{\bbr^{n-1}} e^{-\e |y'|^2 } dy'  \intl_0^\infty s^{\a -1} e^{-\e s}  ds <\infty.\nonumber\eea
Hence, as in (\ref{hyt}), for $\a >0$ we have
$\lang \hat p_{\a -}, \psi \rang= \lim\limits_{\e \to 0} \,\lang (p_{\a -} \eta_\e)^\wedge, \psi \rang$,
where
\bea(p_{\a -} \eta_\e)^\wedge (x)  &=&\frac{1}{\Gam (\a)}\intl_{\bbr^{n-1}}  e^{i x\cdot y -2\e |y'|^2}  dy'\intl^{|y'|^2}_{-\infty}  e^{i x_n  y_n +\e y_n} (|y'|^2 -y_n)^{\a -1} dy_n\nonumber\\
 &=&\frac{1}{\Gam (\a)}\intl_{\bbr^{n-1}} e^{i x'\cdot y'} e^{- |y'|^2 (\e - ix_n)} dy' \intl_0^\infty  s^{\a -1} e^{-s(\e + ix_n)} ds \nonumber\\
 &=&  \pi^{(n-1)/2} (\e \!+ \! ix_n)^{-\a} \,(\e \!- \!ix_n)^{(1-n)/2} \exp \left( \!-\frac{|x'|^2}{4(\e  \!- \! ix_n)}\right ).\nonumber\eea
Here, as above,
\[ (\e \!+ \!ix_n)^{-\a}\!= \!|\e\! + \!ix_n|^{-\a} e^{-i\a\, \arg(\e +ix_n)}, \quad  -\pi/2  <\arg(\e + ix_n) <\pi/2;\]

\[(\e-ix_n)^{(1-n)/2}=|\e-ix_n|^{(1-n)/2} e^{i(1-n)/2\, \arg(\e -ix_n)}, \]

\[-\pi/2  <\arg(\e - ix_n) <\pi/2 .\]
Assuming $\e \to 0$, owing to (\ref {mza}), we obtain
$\lim\limits_{\e \to 0} \, (p_{\a -} \eta_{\e})^\wedge (x)= q_{\a -} (x)$, where
\[
q_{\a -} (x)= \pi^{(n-1)/2} (ix_n)^{-\a}   (-ix_n)^{(1-n)/2}\exp \left(-\frac{i|x'|^2}{4x_n}\right )\]
or, by (\ref{mza}),
\be\label {bdw1}
q_{\a -} (x)= (ix_n)^{-\a-(n-1)/2} \om_{-}(x), \ee
\bea \om_- (x)&=& \pi^{(n-1)/2} \exp \left(\frac{(n-1) \pi i}{2}\, \sgn \,x_n  -\frac{i |x'|^2}{4x_n}  \right )\nonumber\\
&=&\om_+ (x) \exp \left(\!\frac{(n\!-\!1) \pi i}{2}\, \sgn \,x_n\!\right ).\nonumber\eea
 Thus, as in the previous case, we obtain $ (\hat p_{\a -}, \psi)= (q_{\a -}, \psi)$,
as desired.
\end{proof}

The following statement follows from (\ref{mrza}) and Lemma \ref{omor}.
  \begin{lemma} \label {omor2}   The convolution operators $P_{\pm}^{\,\a}$ initially defined by (\ref{poaxz}) for $Re \, \a >0$, extend to all $\a\in \bbc$ as automorphisms of the space $\Phi (\bbr^n)$ by the formula
\be\label  {pbvz}
(P_{\pm}^{\,\a} \vp)(x) = ( q_{\a \pm} \hat\vp)^{\vee}(x), \qquad \vp \in \Phi (\bbr^n).\ee
\end{lemma}

\begin{remark} The formula (\ref {zyqz1}) means that the parabolic Radon transform (\ref{ppar})  is  a composition of the   Riemann-Liouville integral
operator $I_+^{(n-1)/2}$ in the $x_n$-variable and a certain convolution operator  which corresponds to the Fourier multiplier
$\pi^{(n-1)/2}\exp \left(-i |x'|^2/4x_n\right )$ and  is bounded in $L^2(\rn)$; cf. Lemma \ref {ing}  below.
\end{remark}

To compute the  Fourier transforms of the distributions $p^*_{\a \pm}$  corresponding to  the dual fractional integrals , we set
\be\label {bdqap}
 q^*_{\a \pm} (x)\equiv q^*_{\a \pm} (x', x_n) =  q_{\a \pm} (x', -x_n),\ee
or,  by (\ref{bdqa}),
\be\label {bdqaps}
q^*_{\a \pm} (x)=  (\pm ix_n)^{-\a-(n-1)/2}  \om^*_{\pm} (x),\ee
\[
\om^*_+ (x)\!=\!\pi^{(n-1)/2} \exp \left(\frac{i |x'|^2}{4x_n}\right ),\quad \!\om^*_- (x)\!=\!\om^*_+ (x) \exp \left(\!-\frac{(n\!-\!1) \pi i}{2}\, \sgn \,x_n\!\right ).\]

The next statement mimics Lemmas \ref {omor} and \ref {omor2}.

 \begin{lemma} \label{yti1} \hfill

 \noindent {\rm (i)}  For any $\psi\in \Psi (\bbr^n)$ and $\a \in \bbc$,
 \be\label {zazaq } \lang (p^*_{\a \pm})^\wedge, \psi \rang= \lang q^*_{\a \pm}, \psi \rang,\ee
 which means that   $(p^*_{\a \pm})^\wedge=q^*_{\a \pm}$ in the $\Psi'$-sense.
  In particular, if $\a=0$, then for the  delta distribution
(\ref{kss1}) we have
\be\label {zyqzs} \lang (\del^*)^\wedge_P, \vp \rang = \lang q^*_0, \vp \rang,\ee
where
\be\label {zyqz1s} q^*_0 (x)= (ix_n)^{-(n-1)/2} \pi^{(n-1)/2} \exp \left(\frac{i |x'|^2}{4x_n}\right ).\ee

 \noindent {\rm (ii)}  The convolution operators $\stackrel{*}{P}{\!}^\a_{\pm}$ initially defined by (\ref{poax}) for $Re \, \a >0$, extend to all $\a\in \bbc$ as automorphisms of the space $\Phi (\bbr^n)$ by the formula
\be\label  {pbvz}
(\stackrel{*}{P}{\!}^\a_{\pm} \vp)(x) = ( q^*_{\a \pm} \hat\vp)^{\vee}(x), \qquad \vp \in \Phi (\bbr^n).\ee
\end{lemma}

 \begin{corollary} {\rm (cf. (\ref{poeqad}),   (\ref{poeqaq}))} If $\vp\in \Phi (\rn)$, then the equalities
\be\label{ptga}
 P_{\pm}^{\,\a} I_{\pm}^{\,\b} \vp =  I_{\pm}^{\,\b} P_{\pm}^{\,\a} f = P_{\pm}^{\,\a +\b}\vp, \quad \stackrel{*}{P}{\!}^\a_{\pm} I_{\mp}^{\,\b} \vp =  I_{\mp}^{\,\b} \stackrel{*}{P}{\!}^\a_{\pm} \vp = \stackrel{*}{P}{\!}^{\a +\b}_{\pm} \vp\ee
 hold for all $\a, \b \in \bbc$. In particular, for all $\a \in \bbc$,
\be\label{ptga}
 P_{\pm}^{\,\a} \vp=  I_{\pm}^{\,\a} P\vp, \qquad  \stackrel{*}{P}{\!}^\a_{\pm} \vp = I_{\mp}^{\,\a} P^* \vp,\ee
where $P$ and $P^*$ are the Radon transforms (\ref{ppar}),  (\ref{Dacs}).
\end{corollary}

 \begin{corollary}   For all $\a, \b \in \bbc$,
\be\label {wzqs}
  q^*_{\b \pm} (x) \,q_{\a \pm} (x)= \pi^{n-1} |x_n|^{-\a -\b -n +1} \exp \left(\pm\frac{ (\a -\b) \pi i}{2}\,\sgn \, x_n\right ).\ee
 In particular, for $\a=\b$,
\be\label {wzqs4}
  q^*_{\a \pm} (x) \,q_{\a \pm} (x)= \pi^{n-1} |x_n|^{-2\a  -n +1}\ee
and
\be\label {wzqs5}
  q^*_{0} (x) \,q_{0} (x)= \pi^{n-1} |x_n|^{1 -n}.\ee
\end{corollary}

\begin{remark} The meaning of (\ref{wzqs4}) and  (\ref{wzqs5}) is that the composition of the corresponding operators is a constant multiple of the  Riesz potential in the $x_n$-variable. Specifically,
\[
\stackrel{*}{P}{\!}^\a_{\pm} P^\a_{\pm} \vp = P^\a_{\pm} \stackrel{*}{P}{\!}^\a_{\pm} \vp = \pi^{n-1} I_n^{2\a +n -1} \vp, \qquad P^* P \vp =PP^* =\pi^{n-1} I_n^{n -1}\vp,\]
where $\vp \in  \Phi (\bbr^n)$ and, by definition,
\[
(I_n^\lam \vp)^{\wedge}(x) = |x_n|^{-\lam} \hat \vp (x).\]
\end{remark}

\section{Convolutions with  $L^p$-functions}

In this section, we obtain necessary and sufficient conditions, under which  $P^\a_{\pm}$ extend as linear  bounded operators mapping
 $L^p (\rn)$ to $L^q(\rn)$; $1\le p, q\le \infty$.
 First, we investigate the cases $Re \,\a=(1-n)/2$ and  $Re \,\a=1$. Other cases will  be treated by interpolation.

\subsection{The $L^2$-$L^2$ Estimate}

\begin{lemma} \label {ing} For the operators $P^{(1-n)/2 +i\gam}_{\pm}$, $\gam \in \bbr$, initially defined by (\ref{w3a1}),
there exist unique linear bounded extensions
\be\label {llzza} \P^{(1-n)/2 +i\gam}_{\pm}: L^2 (\rn)\to L^2(\rn), \ee
such that for every $f\in L^2 (\rn)$,
\be\label{krccs}
||\P^{(1-n)/2 +i\gam}_{\pm} f||_2 \le \pi^{(n-1)/2} e^{\pi |\gam|/2}\, ||f||_2.\ee
Moreover, in the case $\gam=0$, the operators $\pi^{(1-n)/2}\P^{(1-n)/2}_{\pm}$ are unitary.
\end{lemma}
\begin{proof}  Let us first consider the ``$+$" case and suppress ``$+$" for ease of presentation. By (\ref{mza}) and (\ref{bdqa}),
\[ q_{(1-n)/2 +i\gam} (x)\!=\! \om (x) (-ix_n)^{-i\gam}= \om (x)\exp ( -i\gam \log |x_n| - \frac{\pi \gam}{2}\, \sgn\, x_n),\]
and therefore
\[|q_{(1-n)/2 +i\gam} (x)|=\pi^{(n-1)/2}  e^{-\pi (|\gam|/2)\,{\rm sgn}\, x_n } \le \pi^{(n-1)/2}  e^{\pi |\gam|/2}.\]
Hence there exists a unique linear operator (\ref{llzza})  defined by
\be\label {ent} (\P^{(1-n)/2 +i\gam} f)^\wedge (x)= q_{{(1-n)/2 +i\gam}} (x)\hat f (x), \qquad f \in L^2(\rn),\ee
and satisfying
\[||\P^{(1-n)/2 +i\gam}|| = ||q_{{(1-n)/2 +i\gam}}||_{\infty}\le \pi^{(n-1)/2}  e^{\pi |\gam|/2};\]
 see, e.g.,  Theorem 3.18 in \cite [Chapter I, Section 3] {SW}.

  Let us prove that  $\P^{(1-n)/2 +i\gam}$ is an  extension of  $P^{(1-n)/2 +i\gam}$, that is,
  \[\P^{(1-n)/2 +i\gam} \vp= P^{(1-n)/2 +i\gam}\vp, \qquad \vp \in S(\rn).\]
    Suppose $n$ is even and make use of
  (\ref{Diracs1}) with $\ell =n/2$, $\a=1/2 +i\gam$.
   As in the proof of Lemma \ref{omor} (cf. (\ref{hyt})),  we have
\bea (P^{(1-n)/2 +i\gam} \vp)(x)\! \!\!&=&\! \!\! \lang p_{(1-n)/2 +i\gam} (y), \vp (x \!-\! y) \rang= \lang \partial_n^{n/2} p_{1/2 +i\gam}(y), \vp (x \!- \!y) \rang\nonumber\\
\! \!\! &=& \! \!\!\lang p_{1/2 +i\gam}(y), (\partial_n^{n/2}\vp) (x\! - \!y)\rang\nonumber\\
\! \!\!&=&\! \!\! \lim\limits_{\e \to 0} \,\lang p_{1/2 +i\gam}(y) e_\e (y), (\partial_n^{n/2}\vp) (x \!- \!y)\rang.\nonumber\eea
By Plancherel's formula, setting
\[
g_x (y)= \overline{(\partial_n^{n/2}\vp) (x - y)},\]
we continue:
\bea (P^{(1-n)/2+i\gam} \vp)(x)&=& \lim\limits_{\e \to 0} \,\lang p_{1/2+i\gam}(y) e_\e (y),  \overline{g_x (y)} \,\rang\nonumber\\
&=& (2\pi)^{-n} \lim\limits_{\e \to 0} \,\lang (p_{1/2+i\gam} e_\e)^\wedge (\xi), \overline{\hat g_x (\xi)} \,\rang\nonumber\eea
where
\[
\overline{\hat g_x (\xi)}=\overline{\intl_{\rn} e^{iy\cdot \xi} \, \overline{(\partial_n^{n/2}\vp) (x - y)} \,dy}=  e^{-ix \cdot \xi} (-i \xi)^{n/2} \hat \vp (\xi).\]
Computing $(p_{1/2+i\gam} e_\e)^\wedge (\xi)$, as in (\ref{hytq}),  and passing to the limit under the sign of integration, as in (\ref{hytq1}), we obtain
\bea (P^{(1-n)/2+i\gam} \vp)(x)&=& (2\pi)^{-n}\intl_{\rn}  \hat \vp (\xi) \om (\xi) (-i\xi_n)^{-i\gam} e^{-ix \cdot \xi} d\xi\nonumber\\
&=& (q_{(1-n)/2 +i\gam} \hat \vp)^{\vee} (x).\nonumber\eea
The last expression agrees with (\ref{ent}).
The equality $P_{-}^{(1-n)/2 +i\gam} \vp=\P_{-}^{(1-n)/2 +i\gam} \vp$ can be proved similarly, following the corresponding reasoning in the proof of  Lemma \ref{omor}.
 If $n$ is odd, the proof follows the same lines, but with the choice $\ell =(n+1)/2$ and $\a=1 +i\gam$ while using
 (\ref{Diracs1}).

The operators $\pi^{(1-n)/2}\P^{(1-n)/2}_{\pm}$ are obviously isometries.
To show that  they are unitary, it remains to note that the reciprocals  $(\om_{\pm})^{-1}$  are bounded functions, and therefore they generate  linear bounded operators from  $L^2(\rn)$ to $L^2(\rn)$.
\end{proof}

\begin {remark} The case $\gam =0$ in Lemma \ref{ing}  agrees with  Proposition 2 in \cite {NR} and   Proposition 1 in \cite {BK}, where the reasoning differs from ours.
\end{remark}

\subsection{The $L^1$-$L^\infty$ Estimate}
In the case $\a=1 +i\gam$, $\gam \in \bbr$, we have
\[
(P^{1 +i\gam}_{\pm} \vp)(x)  \frac{1}{\Gam (1 +i\gam )} \intl_{\bbr^n}(y_n -|y'|^2)_{\pm} ^{i\gam}\,  \vp(x-y)\, dy.\]
Note that
\[
|\Gam (1 +i\gam )|^2 =\frac{\pi \gam}{\sinh (\pi \gam)};\]
see, e.g., \cite{AS}. This gives  the following statement.

\begin{lemma} \label {ingq} For any $\gam \in \bbr$ and $\vp \in L^1(\rn)$,
\be\label{krtts}
||P^{1 +i\gam}_{\pm} \vp||_{\infty} \le e^{\pi |\gam|/2}\,  ||\vp||_1.   \ee
\end{lemma}

\subsection{Interpolation}\label {Inter}

The  $L^2$-$L^2$ boundedness of $P_{\pm}^{\,\a}$ in the case $Re \,\a=(1-n)/2$ and the  $L^1$-$L^\infty$ boundedness of these operators with $Re \,\a=1$  on functions $\vp \in S$ pave the way to the  $L^p$-$L^q$ boundedness of $P_{\pm}^{\,\a}$ for intermediate values  of $\a$ via interpolation.
To this end, one may be tempted to  employ Stein's interpolation theorem for analytic families of operators \cite{Ste56}. This theorem can also be found in \cite {BS, Graf04,  SW} and many other sources. However,
the  hypotheses of this theorem and the  proof  are given in terms of simple functions, i.e. finite  linear combinations of the characteristic functions of disjoint compact sets.
It is unclear how to check these hypotheses for our operators, generated by
   distributions and defined on Schwartz functions.

A modification of Stein's  theorem, the hypotheses of which do not contain simple functions, was communicated by Grafakos \cite{Graf}.
In order to  formulate his result, we first make some assumptions and establish notation.

Given $0< p_0, p_1 \le \infty$ and $0<  q_0, q_1 \le \infty$, we  set
\be\label {wzee1}
\frac{1}{p}= \frac{1-\theta}{p_0} + \frac{\theta}{p_1}, \qquad \frac{1}{q}= \frac{1-\theta}{q_0} + \frac{\theta}{q_1}; \qquad 0<\theta <1.\ee
Denote
\[ {\bf S}=\{z \in \bbc : 0<  Re\, z < 1\}, \qquad {\bf \bar S}=\{z \in \bbc : 0\le Re\, z \le 1\}.\]
For $z \in  {\bf S}$, let $T_z$ be a family of linear operators mapping  $C_c^\infty (\rn)$  to $L^1_{loc} (\rn)$ and satisfying the following conditions.

\vskip 0.2 truecm

 (A) For all $\vp, \psi \in C_c^\infty (\rn)$, the function
\be\label {wzws2as}
\A (z)=\intl_{\rn}  (T_z \vp)(x) \psi (x)\, dx \ee
is analytic in  ${\bf S}$ and continuous  on  ${\bf \bar S}$.

\vskip 0.2 truecm

 (B) There exist constants $\gam \in [0, \pi)$ and $s\in (1, \infty]$, such that for any $\vp \in C_c^\infty (\rn)$ and any compact subset $K\subset \rn$,
\be\label {wzws2as1}
\log ||T_z \vp||_{L^s (K)}\le C \, e^{\gam |Im \, z|}\ee
for all $z\in {\bf \bar S}$ and some constant $C=C(\vp, K)$.

\begin {theorem} \label {lanse1} {\rm  \cite [Theorem 5.5.3]{Graf}}
 Let $T_z$, $z \in  {\bf \bar S}$, be a family of linear operators mapping  $C_c^\infty (\rn)$  to $L^1_{loc} (\rn)$ and satisfying  {\rm (A)} and  {\rm (B)} above.
 Suppose that there exist constants  $B_0$, $B_1$, and  continuous functions $M_0 (\gam)$, $M_1 (\gam)$ satisfying
\be\label {wzed1}
M_0 (\gam) + M_1 (\gam) \le \exp (c\, e^{\t |\gam|})\ee
with some constants $c\ge 0$ and $0\le \t <\pi$, such that
 \be\label {wzed} ||T_{i\gam}f ||_{q_0} \le B_0 M_0 (\gam) ||f||_{p_0}, \qquad ||T_{1+i\gam}f||_{q_1} \le B_1 M_1 (\gam) ||f||_{p_1} \ee
for all $\gam \in \bbr$.  Then  for all $f \in C_c^\infty (\rn)$,
\be\label {wzed2} ||T_{\theta}f ||_q  \le B_{\theta} M_{\theta} \,||f||_{p},\ee
where $B_{\theta} =B_0^{1-\theta} B_1^{\theta}$ and
\[ M_{\theta}\!=\!\exp \left \{ \frac{\sin (\pi\theta)}{2} \!\intl_{-\infty}^{\infty} \!\left [ \frac{\log M_0 (\gam)}{\cosh (\pi\gam) \!-\!\cos (\pi\theta)}\! + \!
\frac{\log M_1 (\gam)}{\cosh (\pi\gam)\! +\!\cos (\pi\theta)}\right ] d\gam \right \}.\]
\end{theorem}

\subsection {Proof of Theorem \ref{lanseT}}

The proof consists of two steps, namely, the ``if'' part and the ``only if'' part.

\vskip 0.2 truecm

\noindent {\bf STEP 1.}
We define
\be\label {wzeda}  T_z= P_{\pm}^{\,\a (z)}, \qquad \a(z)=\frac{1+n}{2} z + \frac{1-n}{2}.\ee
Then $0\le Re\, z \le 1$ corresponds to $(1-n)/2 \le Re\, \a \le 1$. In our case,
\[p_0=q_0=2, \qquad p_1=1, \qquad q_1=\infty.\]
If, for  real $\a \in [(1-n)/2, 1]$, we set $\a=((1+n)/2) \theta + (1-n)/2$,  then $\theta=(2\a+n-1)/(n+1)$, and (\ref{wzee1}) yields
\[ 1/p= (\a +n)/(n+1), \qquad 1/q=(1-\a)/(n+1).\]
The latter agrees with (\ref{wz2as}).

To apply  Theorem \ref {lanse1}, we must show that the operator families $T_z= P_{\pm}^{\,\a (z)}$ meet all hypotheses of Theorem \ref {lanse1}.

First, we note that by Lemma \ref{Digf1} (i), $T_z$ maps $S (\rn)$ to $C^\infty (\rn)$, and therefore the required mapping
 $T_z: C_c^\infty (\rn) \to L^1_{loc} (\rn)$ is valid for all $z\in \bbc$, not only for $z\in {\bf S}$.
To  check (A), for any $\vp, \psi \in C_c^\infty (\rn)$, we have
\[
\A (z)=\intl_{\rn}  (T_z \vp)(x) \psi (x)\, dx=\intl_{{\rm supp} (\psi)} \!\!\!F (x,z)\,dx,\]
where
\[ F (x,z)= (T_z \vp)(x) \psi (x)= (P_{\pm}^{\,\a (z)} \vp)(x)\psi (x)\]
is an entire function of $z$, which is $C^\infty$ in the $x$-variable; see Lemma \ref{side}(i) and Lemma \ref{Digf1}(i). This gives (A).

To check (B), we make use of the first equality in (\ref{Ded}) combined with (\ref{008d}). We obtain
\[
(P_{\pm}^{\,\a} \vp)(x)=(\pm 1 )^{k} (P_{\pm}^{\,\a +k} \partial_n^k  \vp)(x)=\frac{(\pm 1 )^{k}}{\Gam (\a +k)}\intl_0^\infty \!s ^{\a+k -1} A^{(k)}_{x,\pm} (s)\,ds,\]
\[
 A^{(k)}_{x,\pm} (s)=\intl_{\bbr^{n-1}} \!\!\!  (\partial_n^k  \vp) (x' -y', x_n \mp s - |y'|^2)\, dy'.\]
Choose $k>  -Re \, \a$ and estimate the obtained expression in the ``+'' case. For any positive integer $m$, there is a constant $c=c(k,m, \vp)$ such that
\[
|A^{(k)}_{x,+} (s)|\le c \intl_{\bbr^{n-1}} \frac{dy'}{(1+  |x' -y'|)^m (1+  |x_n - s - |y'|^2|)^m}.\]
Note that
\be\label{kanzz} \frac{1}{1+  |x' -y'|}\le \frac{1+  |x' -y'|+|x'|}{(1+  |x' -y'|) (1+  |y'|)}\le \frac{1+  |x'|}{1+  |y'|}.\ee
Similarly,
\bea \frac{1}{1+  |x_n - s - |y'|^2|}&=& \frac{1+s+ |y'|^2}{1+  |x_n - s - |y'|^2|}\times  \frac{1}{1+s+ |y'|^2}\nonumber\\
&\le& \left (1+\frac{|x_n|}{1+  |x_n - s - |y'|^2|}\right ) \frac{1}{1+s}
\le \frac{1+  |x_n|}{1+ s}.\nonumber\eea
Hence
\[|A^{(k)}_{x,+} (s)|\le c \left (\frac {(1+  |x'|)(1+  |x_n|)}{1+s}\right )^m \intl_{\bbr^{n-1}} \frac{dy'}{(1+  |y'|)^m }\le c' \,\frac{(1+  |x|)^{2m}}{(1+s)^m}\]
for some constant $c'=c'(k,m, \vp)$ with $k>  -Re \, \a$ and $m>n-1$.
It follows that
\[
|(P_{+}^{\,\a} \vp)(x)|\le \frac{c'\,(1+  |x|)^{2m}}{|\Gam (\a +k)|} \intl_0^\infty \frac{s^{Re \, \a+k -1}}{(1+s)^m} \,ds<\infty\]
provided $-k< Re\, \a < m-k$. In our case, $(1-n)/2 \le Re\, \a \le 1$. Choosing $k> (n-1)/2$ and $m>\, \max (k+1, n-1)$ we make this inequalities compatible, and for any compact set $K\subset \rn$ obtain
\be\label{kanzz1}
||P_{+}^{\,\a} \vp||_{L^\infty (K)}\le \frac{C}{|\Gam (\a +k)|}, \qquad C=C (K, \vp).\ee

 The same inequality can be obtained  for $P_{-}^{\,\a} \vp$ if we slightly change calculations. Specifically, for any  positive integers $m_1$ and $m_2$,
 \[
|A^{(k)}_{x,-} (s)|\le c \intl_{\bbr^{n-1}} \frac{dy'}{(1+  |x' -y'|)^{m_1} (1+  |x_n + s - |y'|^2|)^{m_2}},\]
 where the first factor is estimated as in (\ref{kanzz}) and for the second one we have
 \bea \frac{1}{1+  |x_n + s - |y'|^2|}&=& \frac{1+|s- |y'|^2|}{1+  |x_n + s - |y'|^2|}\times  \frac{1}{1+|s- |y'|^2|}\nonumber\\
&\le& \frac {1+|x_n|}{1+  |s - |y'|^2|}= \frac{(1+|x_n|)(1+s)}{1+  |s - |y'|^2|}\times  \frac{1}{1+s}\nonumber\\
&\le& \frac{(1+|x_n|)(1+ |y'|^2)}{1+ s}.\nonumber\eea
 This gives
\[|A^{(k)}_{x,-} (s)|\le c \,\frac {(1+  |x'|)^{m_1}(1+  |x_n|)^{m_2}}{(1+s)^{m_2}} \intl_{\bbr^{n-1}} \frac{(1+  |y'|^2)^{m_2}}{(1+  |y'|)^{m_1} } dy'.\]
  Choosing $k$, $m_1$, and $m_2$ so that
  \[-k< Re\, \a < m_2-k, \qquad m_1> 2m_2 +n-1,\] we obtain
\[
|(P_{-}^{\,\a} \vp)(x)|\le \frac{c'\,(1+  |x|)^{m_1 +m_2}}{|\Gam (\a +k)|} \intl_0^\infty \frac{s^{Re \, \a+k -1}}{(1+s)^{m_2}} \,ds<\infty.\]
This gives an analogue of (\ref{kanzz1}) for $P_{-}^{\,\a} \vp$. Assuming additionally   $k> (n-1)/2$ and $m_2> k+1$, we make our reasoning compatible with $(1-n)/2 \le Re\, \a \le 1$.

Now, let us revert to the notation of the interpolation theorem.
If $\a=\a (z)$ and $T_z \vp = P_{\pm}^{\,\a(z)} \vp$, then for all $0\le Re\, z \le 1$, the  inequality (\ref{kanzz1}) and its analogue for  $P_{-}^{\,\a} \vp$ give
\[
 ||T_z \vp||_{L^\infty (K)}\le \frac{C}{|\Gam (\z)|}, \qquad \z=\frac{1+n}{2} z + \frac{1-n}{2} +k.\]
If $z=x+iy$, then $\z=a+ib$, where
\[
 a= \frac{1+n}{2}\, x + \frac{1-n}{2} +k, \qquad b=\frac{1+n}{2} \,y; \qquad 0\le x\le 1.\]
It is known that if  $a_1\le a\le a_2$ and $|b| \to \infty$, then
\be\label {quad}
|\Gam (\z)|= \sqrt {2\pi} \, |b|^{a-1/2}\,  e^{-\pi |b|/2}\,  [1 +O (1/|b|)],\ee
where the constant implied by $O$ depends only on $a_1$ and $a_2$; see, e.g., \cite [Corollary 1.4.4]{AAR}.
Taking into account that $1/\Gam (\z)$ is an entire function and using (\ref{quad}), after simple calculations we obtain an estimate of the form
\[ \log ||T_z \vp||_{L^{\infty} (K)}\le \left \{
\begin{array} {ll} \! c_1&\mbox{ if $|y|\le 10$},\\
c_2 +c_3 \log |y| +c_4 |y|&\mbox{ if $|y|\ge 10$},\\
 \end{array}
\right.\]
for some positive constants $c_1, c_2, c_3, c_4$.
This estimate yields (\ref {wzws2as1})  for any $\gam >0$. Thus
 verification of the assumption (B) for Theorem \ref {lanse1} is complete.

Let us check (\ref{wzed1}). By (\ref{wzeda}),
\[  T_{i\gam}=P_{\pm}^{(1-n)/2 +i\gam (1+n)/2}, \qquad  T_{1+i\gam}=P_{\pm}^{1 +i\gam (1+n)/2}.\]
Hence the results of Lemmas \ref{ing} and \ref{ingq} can be stated as
\[ ||T_{i\gam}\vp ||_{2} \le B_0 M_0 (\gam) ||\vp||_{2}, \qquad ||T_{1+i\gam} \vp||_{\infty} \le B_1 M_1 (\gam) ||\vp||_{1}, \]
\[ M_0 (\gam)=M_1 (\gam)=\exp (\pi (1+n)|\gam|/4),\]
with some constants $B_0, B_1$. These estimates yield (\ref{wzed1}).

Thus, by Theorem \ref {lanse1}, the ``if'' part of Theorem \ref{lanseT} is proved for $P^\a_{\pm}$ when $\a$ is real.
If $\a$ has a nonzero  imaginary part, say, $t$, the above reasoning can be repeated almost verbatim if we re-define $T_z$ in (\ref{wzeda})  by setting $T_z= P_{\pm}^{\,\a (z+it)}$.


\vskip 0.2 truecm

\noindent {\bf STEP 2.} Let us prove the ``only if'' part of Theorem \ref{lanseT}.
First we show that the left bound $Re \, \a= (1-n)/2$ in (\ref{wz2as}) is sharp.
Suppose the contrary,  assuming for simplicity that $\a$ is real.  Then there is  a triple $(p_0, q_0, \a_0)$ with $1\le p_0< q_0 \le \infty$ and $\a_0<(1-n)/2$, such that  at least one of the operators $P^{\a_0}_{\pm}$, say $P^{\a_0}_{+}$, is bounded from  from $L^{p_0}(\rn)$ to $L^{q_0}(\rn)$.
Then, interpolating the  triples  $(p_0, q_0, \a_0)$ and $(1, \infty, 1)$, as we did above, we conclude that for any $\a \in [\a_0, 1]$, the operator $P^\a_{+}$ is bounded from  from $L^{p_\a}(\rn)$ to $L^{q_\a}(\rn)$, where
\[ \frac{1}{p_\a}= \frac{\a -\a_0}{ 1 -\a_0} + \frac{1-\a }{ p_0 (1 -\a_0)}, \qquad  \frac{1}{q_\a}=  \frac{1-\a }{ q_0 (1 -\a_0)}.\]
 In particular, for $\a= (1-n)/2$ and $\a_0=(1-n)/2 -\e$, $0<\e< (1+n)/2$, we obtain
\[ p_{(1-n)/2}=\frac{2(1+n)}{1+n -2\e}, \qquad q_{(1-n)/2}=\frac{2(1+n)}{1+n +2\e}.\]
The latter agrees with the known $L^2$-$L^2$ boundedness of $P^{(1-n)/2}_{+}$ only if $\e=0$, that is, $\a_0=(1-n)/2$.

The necessity of $p$ and $q$ in (\ref{wz2as}) can be proved using the  scaling argument, which is applied to the equivalent operator families
\be\label {99jh1}
 (T_{\pm}^\a f)(x= \frac{1}{\Gam (\a)}\intl_{\rn} (x_n -y_n)_{\pm}^{\a -1} f(y', y_n + x' \cdot y')\, dy, \quad  Re \, \a> 0.\ee
The limiting case $\a=0$ in (\ref{99jh1}) yields the  transversal Radon transform
 (\ref{bart}). There is a  remarkable connection between  $P^\a_{\pm}f$ and $T_{\pm}^\a f$, which  is realized by  the maps
\be \label {barsd}
(B_1f)(x)=f(x', x_n \!-\!|x'|^2), \quad (B_2 F)(x)=F(2x',  x_n \!-\!|x'|^2).\ee
Their  inverses  have the form
\be \label {bars1} (B_1^{-1}u)(x) \! = \!u(x', x_n +|x'|^2), \quad (B_2^{-1} v)(x)\!=\!v \left (\frac{x'}{2},  x_n +\frac{|x'|^2}{4}\right ).\ee
One can easily show that
\be\label {iar} ||B_1f||_p =||f||_p, \qquad \| B_2 F \|_{q} = 2^{(1-n)/q} \,\|F \|_{q}.\ee

\begin{lemma}\label {swawd} {\rm \cite [Lemma 7.5] {Ru22a}}  The equality
\be \label {btsw}  P_{\pm}^{\,\a} f=B_2 T_{\pm}^\a B_1 f, \qquad Re \,\a >0,\ee
holds,  provided that either side  of it exists in the Lebesgue sense. If $f\in S(\rn)$, then (\ref{btsw}) extends to all $\a \in\bbc$ by analytic continuation.
\end{lemma}
\begin{proof}  We recall the proof for the sake of completeness. In the ``$+$'' case we have
\bea
&&(B_2^{-1}P_{+}^{\,\a} f)(x)\!= \!\frac{1}{\Gam (\a)} \intl_{\bbr^n}(y_n \!-\!|y'|^2)_{+} ^{\a -1} \,f \left( \frac{x'}{2}\!- \!y', x_n \!+\! \frac{|x'|^2}{4} \! -\! y_n \right) dy,\qquad\nonumber\\
&& =\frac{1}{\Gam (\a)}  \intl_{\bbr^{n-1}} dy' \intl_0^\infty s^{\a -1} f \left( \frac{x'}{2}- y', x_n + \frac{|x'|^2}{4}  - s- |y'|^2\right ) ds\nonumber\\
\label {mmnz}&& =\frac{1}{\Gam (\a)}  \intl_0^\infty s^{\a -1} ds  \intl_{\bbr^{n-1}}  f (z', x_n -s + x'\cdot z'- |z'|^2)\, dz'.\eea
On the other hand,
\bea
&&(T_{+}^\a B_1f)(x)=\frac{1}{\Gam (\a)}\intl_{\rn} (x_n -y_n)_{+}^{\a -1} (B_1f)(y', y_n + x' \cdot y')\, dy\nonumber\\
&&= \frac{1}{\Gam (\a)}  \intl_0^\infty s^{\a -1} ds \intl_{\bbr^{n-1}}  f (y',  x_n -s + x'\cdot y'- |y'|^2)\, dy',\nonumber\eea
which coincides with (\ref{mmnz}). In the ``$-$'' case, the proof is similar.
\end{proof}

An analogue of (\ref{btsw}) for the Radon transforms (\ref{ppar}) and (\ref{bart}) can be found in \cite [Lemma 2.3]{Chr} and \cite[Lemma 3.2]{Ru22}.

Let us continue the proof of the main theorem.
 In view of (\ref{barsd}) and
 Lemma \ref{swawd}, it suffices to work  with   $T^\a_{\pm}$ in place of  $P^\a_{\pm}$, where the case $\a=0$ corresponds to the Radon transform  (\ref{bart}). Let, for example $\a$ be real.   For $\lam =(\lam_1, \lam_2)$, $\lam_1 >0$, $\lam_2 >0$,  denote
\[
(A_\lam f)(x)=f (\lam_1 x', \lam_2 x_n), \qquad (B_\lam F)(x)=\frac{\lam_1^{1-n}}{\lam_2^\a } F\left (\frac{\lam_2}{\lam_1}\, x', \lam_2 x_n\right ). \]
Then $T_{\pm}^\a A_\lam f = B_\lam T_{\pm}^\a f$ and we have
\[ ||A_\lam f||_p=\lam_1^{(1-n)/p}\lam_2^{-1/p} ||f||_p, \qquad ||B_\lam F||_q=\lam_1^{1-n+(n-1)/q}\lam_2^{-\a-n/q} ||F||_q.\]
If $|| T^\a_{\pm} f||_q \le c \,||f||_p$ is true for all $f\in L^p$, then it is true for $A_\lam f$, that is, $||T^\a_{\pm} A_\lam f||_{q}\le c \,||A_\lam f||_{p}$ or
$||B_\lam T_{\pm}^\a f||_{q}\le c \,||A_\lam f||_{p}$.
 The latter is equivalent to
\[
\lam_1^{1-n+(n-1)/q}\lam_2^{-\a-n/q} ||T_{\pm}^\a f||_q   \le c \,\lam_1^{(1-n)/p}\lam_2^{-1/p} ||f||_p.\]
Assuming that $\lam_1 $ and $\lam_2 $ tend to zero and to infinity, we conclude that the last inequality is possible only if
\be\label {hghfv}
p=\frac{n+1}{n + \a}, \qquad q=\frac{n+1}{1- \a}.\ee
The above  reasoning  shows that the right bound $Re \,\a= 1$ is sharp, too.

Now, the proof of  Theorem \ref{lanseT} is complete.

\section{Conclusion}

Some comments are in order.

\vskip 0.2 truecm

\noindent  {\bf 1.} In the present article we explored parabolic convolutions associated with the Riemann-Liouville operators   (\ref{poqa}). The same method can be applied to convolutions of the Riesz type
\[
(P^{\,\a} f)(x)=\frac{1}{\gamma (\a)} \intl_{\bbr^n}|y_n -|y'|^2|^{\a -1} \,f(x-y)\, dy,\]
\[
(P_s^{\,\a} f)(x)=\frac{1}{\gamma' (\a)} \intl_{\bbr^n}|y_n -|y'|^2|^{\a -1}\, \sgn (y_n -|y'|^2)\,f(x-y)\, dy,\]
where
\[ \gamma (\a)= 2\Gamma(\a) \,\cos (\a\pi/2), \qquad \gamma'(\a)=2i\Gam (\a) \sin (\a \pi/2);\]
cf.  \cite [Chapter I, Section 3]{GS1},  \cite {SKM}.

\vskip 0.2 truecm

\noindent  {\bf 2.}  Theorem \ref {lanseT} guarantees the existence of the $L^p$-$L^q$ bounded extensions of the operators $P_{\pm}^{\,\a}$ provided by  interpolation. We denote these extensions by $\P_{\pm}^{\,\a}$. It is natural to ask:

 {\it What are the  explicit analytic formulas for the $L^q$-functions $\P_{\pm}^{\,\a}f$?}

The case $Re\, \a \le 0$ is especially intriguing because our operators are not represented by absolutely convergent integrals and need a suitable $L^q$-regularization.
  A similar question for solutions of the wave equation was studied in \cite{Ru89}.

  Using the same reasoning as in \cite{Ru22a}, one can show that in the cases $\a=0$ (for the Radon transform $P$) and $0< Re \, \a <1$, we have $(\P_{\pm}^{\,\a}f)(x)=(P_{\pm}^{\,\a}f)(x)$ for almost $x$. If $(1-n)/2 \le Re\, \a \le 0$  with $Im \, \a \neq 0$, the expression $(\P_{\pm}^{\,\a}f)(x)$ can be represented by  hypersigular integrals, converging in the $L^q$-norm and in the a.e. sense. We leave the details to the interested reader.

 \vskip 0.2truecm

\noindent  {\bf 3.}  Analytic families of fractional integral of convolution type arise in the context of non-Euclidean harmonic analysis, when the concept of the convolution is determined by the corresponding Lie group of motions. Examples of such convolutions on the unit sphere and the hyperbolic space can be found, e.g., in \cite {Ru15, Str70, Str81}. It might be of interest to adjust the hypotheses of  Stein's interpolation theorem from \cite{Ste56}
for these cases and obtain the corresponding $L^p$-$L^q$ estimates (or complete the existing proofs).

\vskip 0.2 truecm

\noindent {\bf Acknowledgement.} The author is grateful to Professor Loukas Grafakos for useful discussions and sharing his knowledge of the subject.

\section{Appendix}

Let us  show that the functions  $p_{\a \pm}$, $Re\, \a >0$, can be viewed as regular   tempered distributions.
Given a test function  $\vp \in S (\rn)$, for any positive integers $\ell$ and $m$ there is a constant $c_{\ell, m}$ such that
\[ \sup\limits_{y}\, (1+|y'|)^\ell  (1+|y_n|)^m |\vp (y', y_n)|\le c_{\ell, m}.\]
Hence, setting $\a_0= Re \, \a >0$ and
passing to polar coordinates, for sufficiently large  $\ell$ and $m$ we have
\bea
|\lang p_{\a +}, \vp \rang| &\le& \frac{1}{|\Gam (\a)|} \intl_0^\infty r^{n-2} dr \intl_{r^2}^\infty  (y_n - r^2)^{\a_0 -1} dy_n \intl_{S^{n-2}} |\vp (r \theta, y_n)| d\theta\nonumber\\
&\le& \frac{\sig_{n-2 }\,c_{\ell, m}}{|\Gam (\a)|} \intl_0^\infty \frac{r^{n-2}\, dr}{(1+r)^\ell} \intl_{r^2}^\infty \frac{(y_n - r^2)^{\a_0 -1}}{(1+y_n)^m}\, dy_n \nonumber\\
&\le& \frac{\sig_{n-2 }\,c_{\ell, m}}{|\Gam (\a)|}  \intl_0^\infty \frac{r^{n-2}\, dr}{(1+r)^\ell}
 \intl_{0}^\infty \frac{s^{\a_0 -1}}{(1+s)^m}\, ds <\infty.\nonumber\eea
For $p_{\a -}$, the calculations are a little bit more sophisticated. We have
\bea
|\lang p_{\a -}, \vp \rang| &\le& \frac{1}{|\Gam (\a)|} \intl_0^\infty r^{n-2} dr \intl^{r^2}_{-\infty}  (r^2 -y_n)^{\a_0 -1} dy_n \intl_{S^{n-2}} |\vp (r \theta, y_n)| d\theta\nonumber\\
&\le& \const  \intl_0^\infty \frac{r^{n-2}\, dr}{(1+r)^\ell} \intl_{0}^\infty \frac{s^{\a_0 -1}}{(1+ |r^2 -s|)^m}\, ds \le \const (I_1 +I_2).\nonumber\eea
Here
\[
I_1 = \intl_0^\infty \frac{r^{n-2}\, dr}{(1+r)^\ell} \intl_{0}^{r^2} \frac{s^{\a_0 -1}}{(1+ r^2 -s)^m}\, ds \le   \intl_0^\infty \frac{r^{n-2}\, dr}{(1+r)^\ell} \intl_{0}^{r^2} s^{\a_0 -1} ds <\infty\]
if $\ell$ is large enough. Further,
\[
I_2= \intl_0^\infty \frac{r^{n-2}\, dr}{(1+r)^\ell} \intl^{\infty}_{r^2} \frac{s^{\a_0 -1}}{(1+ s-r^2)^m}\, ds=
\intl_0^\infty \frac{r^{n-2}\, dr}{(1+r)^\ell} \intl^{\infty}_{0} \frac{ (r^2 +t)^{\a_0 -1}}{(1+ t)^m}\, dt.\]
If $\a_0 \le 1$, then
\[
I_2\le \intl_0^\infty \frac{r^{n-2}\, dr}{(1+r)^\ell} \intl^{\infty}_{0} \frac {dt}{t^{1-\a_0}\,(1+ t)^m} <\infty.\]
If $\a_0 > 1$, then $I_2=I_{2,1} + I_{2,2}$, where $I_{2,1} = \intl_0^1 (...)\, dr$,  $I_{2,2} = \intl_1^\infty (...) \,dr$,
\[
I_{2,1} \le  \intl_0^1 \frac{r^{n-2}\, dr}{(1+r)^\ell}  \intl^{\infty}_{0}  (1+t)^{\a_0 -1-m} dt <\infty,\]
\bea
I_{2,2} &\le&  \intl_1^\infty \frac{r^{n-2}}{(1+r)^\ell} \left (\intl^{r}_{0}  (r^2 +t)^{\a_0 -1} dt + \intl_{r}^{\infty} \frac{(r^2 +t)^{\a_0 -1}}{t^m} dt \right )\nonumber\\
&=&  \intl_1^\infty \frac{r^{n-2}}{(1+r)^\ell} \, (c_1 r^{2\a_0} + c_2 r^{2\a_0 - m} )\, dr <\infty.\nonumber\eea
Thus $p_{\a -} \in S' (\rn) $ for all $ Re \, \a >0$.

\end{document}